\documentclass[12pt]{amsart}
\usepackage{amsfonts}
\usepackage{amsmath}
\usepackage{amssymb,latexsym}
\usepackage[mathcal]{eucal}
\usepackage{amscd}
\usepackage{hyperref}

\input xy
\xyoption{all}

\newcommand{\Acal}{\mathcal{A}}
\newcommand{\Bcal}{\mathcal{B}}

\newcommand{\Pcal}{\mathcal{P}}

\newcommand{\N}{\mathbb{N}}
\newcommand{\Z}{\mathbb{Z}}

\newcommand{\E}{\mathbb{E}}

\newcommand{\al}{\alpha}
\newcommand{\ga}{\gamma}

\newcommand{\del}{\delta}

\newcommand{\ep}{\epsilon}
\newcommand{\sig}{\sigma}

\newcommand{\om}{\omega}
\newcommand{\Om}{\Omega}

\newcommand{\id}{{\rm{id}}}

\newtheorem{theorem}{Theorem}[section]
\newtheorem{lemma}[theorem]{Lemma}
\newtheorem{prop}[theorem]{Proposition}

\theoremstyle{definition}
\newtheorem{definition}[theorem]{Definition}

\theoremstyle{remark}

\numberwithin{equation}{section}

\begin{document}

%%%%%%%%%%%%%%%%%%%%%%%%%%%%%%%%%%%%%%%%%%%%%%%%%%%%%%%%%%%%%%%%%%%%
%%%%%%%%%%%%%%%%%%% T H E    T I T L E    %%%%%%%%%%%%%%%%%%%%%%%%%%
%%%%%%%%%%%%%%%%%%%%%%%%%%%%%%%%%%%%%%%%%%%%%%%%%%%%%%%%%%%%%%%%%%%%

\title
{Recurrence for stationary group actions}

\author{Hillel Furstenberg }

\address{Department of Mathematics\\
     Hebrew University of Jerusalem\\
         Jerusalem\\
         Israel}
\email{harry@math.huji.ac.il}

\author{Eli Glasner }

\address{Department of Mathematics\\
     Tel Aviv University\\
         Ramat Aviv\\
         Israel}
\email{glasner@math.tau.ac.il}

\date {September 21, 2011}



\dedicatory{In memory of Leon Ehrenpreis.}

\begin{abstract}
Using a structure theorem from \cite{FG} we prove a version of multiple recurrence for sets of positive measure in a general stationary dynamical
system.
\end{abstract}

\maketitle

%%%%%%%%%%%%%%%%%%%%%%%%%%%%%%%%%%%%%%%%%%%%%%%%%%%%%%%%%%%%%%%%%%%%%
%%%%%%%%%%%%%%%%%%%%%%%%  CONTENTS   %%%%%%%%%%%%%%%%%%%%%%%%%%%%%%%%
%%%%%%%%%%%%%%%%%%%%%%%%%%%%%%%%%%%%%%%%%%%%%%%%%%%%%%%%%%%%%%%%%%%%%

\tableofcontents \setcounter{secnumdepth}{1}

%\addtocontents{toc}{subsection}{\protect\hspace{0.5cm}}
%\newpage

%%%%%%%%%%%%%%%%%%%%%%%%%%%%%%%%%%%%%%%%%%%%%%%%%%%%%%%%%%%%%%%%%%%%%
%%%%%%%%%%%%%%%%%%%%%%%%%%%%  TEXT   %%%%%%%%%%%%%%%%%%%%%%%%%%%%%%%%
%%%%%%%%%%%%%%%%%%%%%%%%%%%%%%%%%%%%%%%%%%%%%%%%%%%%%%%%%%%%%%%%%%%%%

\setcounter{section}{0}
\setcounter{page}{1}

%%%%%%%%%%%%%%%%%%%%%%%%%%%%%%%%%%%%%%%%%%%%%%%%%%%%%%%%%%%%%%%%%%%%%
%%%%%%%%%%%%%%%%%%%       Inroduction         %%%%%%%%%%%%%%%%%%%%%%%
%%%%%%%%%%%%%%%%%%%%%%%%%%%%%%%%%%%%%%%%%%%%%%%%%%%%%%%%%%%%%%%%%%%%%

\section*{Introduction}

The celebrated theorem of E. Szemer\'{e}di regarding the existence of
long arithmetical progressions in subsets of the integers having positive (upper) density is known to be equivalent to a statement involving ``multiple recurrence''
in the framework of dynamical systems theory (see e.g. \cite{Fur5}). Stated precisely this is the following assertion:

\begin{theorem}
Let $(X,\Bcal, \mu,T)$ be a measure preserving dynamical system; i.e.,
$(X,\Bcal, \mu)$ is a probability space and $T : X \to X$ a measure preserving
mapping of $X$ to itself. If $A \in \Bcal$ is a measurable subset with
$\mu(A) >0$, then for any $k=1,2,3,\dots$ there exists $m \in \N=\{1,2,3,\dots\}$ with
$$
\mu(A \cap T^{-m}A \cap T^{-2m}A \cap \cdots \cap T^{-km}) > 0.
$$
\end{theorem}

The case $k=1$ is the ``Poincar\'{e} recurrence theorem'' and is an easy
exercise in measure theory. The general case is more recondite
(see e.g. \cite{Fur5}). In principle recurrence phenomena make sense in the
framework of more general group actions and we can inquire what is the
largest domain of their validity. Specifically if a group $G$ acts on a measure
space $(X,\nu)$ (we have suppressed the $\sigma$-algebra of measurable sets)
with $(g,x) \to T_gx$ by non-singular maps $\{T_g\}$, and $A$ is a measurable subset of  $X$ with $\nu(A) >0$, under what conditions can we find for large
$k$ an element $g \in G$, \ $g \not=$ identity, with
$$
\nu(A \cap T_g^{-1}A \cap T_g^{-2}A \cap \cdots \cap T_g^{-k}A) > 0\ ?
$$
Some conditions along the line of measure preservation will be necessary.
Without this we could take $G =\Z, \ X = \Z\cup \{\infty\},\
\ \forall\ t, n \in \Z,\ T_t n = n +t,\
n \not = \infty, \ T_t \infty = \infty$, and
$\nu(\{n\})= \frac{1}{3\cdot 2^{|n|}}$,\  $\nu(\{\infty\})=0$. Here no $t \not= 0$
will satisfy $T_t(\{n\}) \cap \{n\} \not= 0$.

The present work extends an earlier paper on ``stationary'' systems
(\cite{FG}). Here we shall show that quite generally, under the hypothesis
of ``stationarity'', which we shall presently define, one obtains a version of
multiple recurrence for sets of positive measure.

We recall the basic definitions here, although we will rely on the
treatment in \cite{FG} for fundamental results. Throughout $G$
will represent a locally compact, second countable group, and
$\mu$ a fixed probability measure on Borel sets of $G$. We
consider measure spaces $(X,\nu)$ on which $G$ acts measurably,
i.e. the map $G \times X \to X$ which we denote $(g,x) \to gx$ is
measurable, and so the convolution of the measure $\mu$ on $G$ and
$\nu \in \Pcal(X),\ \mu*\nu$, is defined as the image of $\mu
\times \nu$ on $X$ under this map; thus $\mu*\nu$ is again a
probability measure on $X$, an element of $\Pcal(X)$. { \bf We
will always assume that $G$ acts on $(X,\nu)$ by non-singular}
transformations; i.e., $\nu(A) =0$ implies $\nu(gA)=0$ for a
measurable $A \subset X$ and $g \in G$.

\begin{definition}
When $\mu*\nu =\nu$ we say that $(X,\nu)$ is a stationary $(G,\mu)$
space.
\end{definition}

This can be interpreted as saying that $\nu$ is invariant ``on the average''.
It is also equivalent to the statement that for measurable $A \subset X$
\begin{equation}\label{av}
\nu(A) = \int_G \nu(g^{-1}A)\, d\mu(g).
\end{equation}

Associated with the space $(G,\mu)$ we will consider the probability space
$$
(\Omega,P) = (G,\mu) \times (G,\mu) \times \cdots
$$
where we will denote the random variables representing the
coordinates of a point $\om \in \Om$ by
$\{\xi_1(\om), \xi_2(\om),\dots,\xi_n(\om),\dots \}$.
We will draw heavily on the ``martingale convergence theorem''
which for our purposes can be formulated:

\begin{theorem}[Martingale convergence theorem (MCT)]
Let $\{F_n(\om)\}_{n \in \N}$ be a sequence of uniformly bounded, measurable,
real valued functions on $\Om$ 
% bbb
with $F_n$ measurable with respect to
$\xi_1,\xi_2,\dots,\xi_n$ and such that
\begin{equation}
F_n(\xi_1,\xi_2,\dots,\xi_n) =
\int_G F_{n+1} (\xi_1,\xi_2,\dots,\xi_n,\eta)\, d\mu(\eta).
\end{equation}
(Such a sequence is called a {\em martingale}.)
% and set $f_n(\om) = F_n(\xi_1(\om), \xi_2(\om), \dots ,\xi_n(\om))$. 
Then with probability one, the sequence $\{F_n(\om)\}$
converges almost surely to a limit $F(\om)$ satisfying:
\begin{equation}\label{limit}
\E(F) = \int F(\om) \, dP(\om) = \int_G F_1(\eta)\, d\mu(\eta).
\end{equation}
\end{theorem}

The theory of stationary actions is intimately related to boundary
theory for topological groups and the theory of harmonic
functions. For details we refer the reader to \cite{F73}.

\section{Poincar\'{e} recurrence for stationary actions}\label{sec:po}

A first application will be a proof of a particular version of the
Poincar\'{e} recurrence phenomenon for stationary actions.

\begin{theorem}\label{thm:Po}
Let $G$ be an infinite discrete group and let $\mu$ be a probability measure
on $G$ whose support $S(\mu)$ generates $G$ as a group. Let $(X,\nu)$ be
a stationary space for $(G,\mu)$ and let $A \subset X$ be a measurable
subset with $\nu(A) >0$. Then there exists $g \in G, \ g \not=$ identity,
with $\nu(A \cap g^{-1}A) > 0$.
\end{theorem}

We start with a lemma.

\begin{lemma}
If $\Sigma(\mu)$ is the semigroup in $G$ generated by $S(\mu)$,
there exists a sequence of elements $\al_1,\al_2, \al_3, \dots \in \Sigma(\mu)$
such that no product $\al_{i_1}\al_{i_2}\cdots\al_{i_n}$ with
$i_1 < i_2 < i_3 < \cdots <i_n$ equals the identity element of $G$.
\end{lemma}

\begin{proof}
The semigroup $\Sigma(\mu)$ is infinite since a finite subsemigroup
of a group is a group. We proceed inductively so that having defined
$\al_1,\al_2,\dots,\al_n$ where products don't degenerate, we can find
$\al_{n+1} \in \Sigma(\mu)$ so that no product
$$
\al_{i_1}\al_{i_2} \cdots \al_{i_s} \al_{n+1} \allowbreak = \id,
$$
there being only finitely many values to avoid.
\end{proof}

\begin{proof}[Proof of the theorem \ref{thm:Po}:]
The proof is based on two ingredients. First, if we define functions on $\Om$ by
$$
F_n(\om) = \nu(\xi_n^{-1}\xi_{n-1}^{-1}\cdots\xi_1^{-1}A)
$$
then by (\ref{av}), the sequence $\{F_n\}$ forms a martingale. The
second ingredient is the fact that in almost every sequence
$\xi_1(\om), \xi_2(\om), \dots, \allowbreak \xi_n(\om), \dots$
every word in the ``letters" of $S(\mu)$ appears infinitely far
out, and then every element in $\Sigma(\mu)$ appears as a partial
product. Now let $f(\om) = \lim F_n(\om)$, which by the MCT is
defined almost everywhere, then $\E(f) = \int \nu(g^{-1}A)\,
d\mu(g)=\nu(A) >0$. So, if $\del =\nu(A)/2$, there will be a
random variable $n(\om)$ which is finite with positive probability
so that for $n > n(\om)$
$$
\nu(\xi_n^{-1}\xi_{n-1}^{-1}\cdots\xi_1^{-1}A) > \del.
$$
Now choose $\al_1,\al_2, \al_3, \dots \in \Sigma(\mu)$  as in the foregoing
lemma, and let $N > 1/\del$.
With positive probability there is $l \ge n(\om)$ and
$0 =r_0 < r_1 < r_2 < \cdots < r_N$ so that in $\Sigma(\mu)$,
$$
\xi_{l + r_{i-1} +1}(\om)\xi_{l + r_{i-1} +2}(\om)\cdots \xi_{l + r_i}(\om) =\al_i,
$$
for $i=1,2,\dots,N -1$.
By definition of $n(\om)$
$$
\nu(\al_i^{-1}\cdots\al_l^{-1}\beta^{-1} A) > \del \qquad i=1,2,\dots,N,
$$
where $\beta = \xi_1\xi_2 \cdots\xi_{l-1}$.
But this yields $N$ sets of measure $> 1/N$ in $X$ and we conclude that
for some $ i < j$
$$
\nu(\al_i^{-1}\cdots\al_l^{-1}\beta^{-1} A \cap
\al_j^{-1}\cdots\al_l^{-1}\beta^{-1} A) > 0.
$$
This however implies that for a conjugate $\gamma$ of the product
$\al^{-1}_j\al_{j-1}^{-1}\cdots\al_{i+1}^{-1}$ we have
$\nu(A \cap \ga A) >0$. Here $\gamma \not = \id$ since by
construction $\al_{i+1}\al_{i+2}\cdots\al_j \allowbreak \not= \id$.
\end{proof}

\section{Multiple recurrence for SAT actions}\label{sec:SAT}

Our main result is a multiple recurrence theorem for stationary actions.
We proceed step by step proving the theorem
% aaa
first for the special category of actions known as SAT actions. These were introduced by Jaworski in \cite{Ja}.

\begin{definition}
The action of a group $G$ on a probability measure space $(X,\nu)$ is
SAT (strongly approximately transitive) if for every measurable $A \subset X$
with $\nu(A) >0$, we can find a sequence $\{g_n\}\subset G$ with
$\nu(g_nA) \to 1$.
\end{definition}

We now have a second recurrence result:

\begin{theorem}\label{thm:SAT}
%If $(X,\nu)$ is a stationary $(G,\mu)$ space for a SAT action of $G$
%on $(X,\nu)$,
If $(X,\nu)$ is a probability measure space on which the group $G$
acts by non-singular transformations and the $G$ action is SAT,
then for every measurable $A \subset X$ with $\nu(A) >0$
and any integer $k \ge 1$, there is a $\gamma \in G,\ \gamma \not=\id$ with
\begin{equation}\label{cat}
\nu(A \cap \gamma^{-1}A \cap \gamma^{-2}A \cap
\cdots \cap \gamma^{-k}A) > 0\ .
\end{equation}

Moreover if $F$ is any finite subset of $G$, $\gamma$ can be chosen outside of $F$.
\end{theorem}

We use the following basic lemma from measure theory.

\begin{lemma}\label{lem:ac}
If $\sig : X \to X$ is a non-singular transformation with respect to a
measure $\nu$ on $X$, then for any $\ep > 0$, there exists a $\del >0$
so that $\nu(A) < \del$ implies $\nu(\sig A) < \ep$.
\end{lemma}

\begin{proof}
If such a $\del$ did not exist we could find $B \subset X$ with $\nu(B)=0$
and $\nu(\sig B) \ge \ep$.
\end{proof}

\begin{prop}\label{prop:finite}
Assume $G$ acts on $(X,\nu)$ by non-singular transformations
and let $\ga_1, \ga_2, \dots,\ga_k \in G$. There exists $\del >0$ so that if
$\nu(B) > 1 - \del$ then
$$
\nu( \ga_1 B \cap \ga_2 B \cap \cdots \cap \ga_k B) > 0.
$$
\end{prop}

\begin{proof}
The desired inequality will take place provided the measure of each
$\ga_i B'$ is less than $1/k$, where $B' = X \setminus B$.
By Lemma \ref{lem:ac} this will hold if $\nu(B')$ is sufficiently small.
\end{proof}

\begin{proof}[Proof of the theorem \ref{thm:SAT}:]
Let $\sig \not= \id$ be any element of $G$. Apply Proposition \ref{prop:finite}
with 
% ggg
$\ga_0 = \id$, 
$\ga_i = \sig^{-i},\ i=1,2,\dots ,k$ and find $\del > 0$ so that $\nu(B)
> 1 -\del$ implies
$$
\nu(B \cap \sig B \cap \sig^2 B \cap
\cdots \cap \sig^k B) > 0\ .
$$
Use the SAT property to find $g \in G$  with $\nu(g^{-1}A) > 1 -\del$.
Then
$$
\nu(g^{-1}A \cap \sig g^{-1}A \cap \sig^2 g^{-1}A \cap
\cdots \cap \sig^k g^{-1}A) > 0.
$$
Applying $g$ to the set appearing here we get:
$$
\nu(A \cap g \sig g^{-1}A \cap  g \sig^2 g^{-1}A \cap
\cdots \cap g \sig^k g^{-1}A) > 0.
$$
Letting $\ga= g  \sig^{-1} g^{-1}$ we obtain the desired result.

We turn now to the last statement of the theorem.  One sees easily that if $G$ has a non-trivial SAT action then $G$ is infinite.  Let $H$ be a finite subset of $G$ with greater cardinality than $F$.  Now in the foregoing discussion we consider a sequence $\{ g_n\}$ in $G$ with $\nu (g_n^{-1} A) \to 1$; then for any $\sigma$, if $n$ is sufficiently large if we take 
% ggg
$\ga = g_n \sigma g_n^{-1}$ for large $n$ we will get \eqref{cat}. We claim that $\sigma$ can be chosen so that for an infinite subsequence $\{n_j\}$ we will have $g_{n_j} \sigma g_{n_j}^{-1} \notin F$. For this we simply consider the sets $\{ g_n Hg_n^{-1}\}$ each of which has some element outside of $F$. $\{ n_j\}$  is then a sequence for which there is a fixed $\sigma \in H$ with $g_{n_j} \sigma g_{n_j}^{-1} \notin F$.  This completes the proof.
\end{proof}

\section{A Structure theorem for Stationary Actions}

In order to formulate our structure theorem we will introduce a few definitions
and some well known basic tools from the general theory of dynamical systems.

\subsection{Factors and the disintegration of measures}
\begin{definition}
Let $(X, \nu)$ and $(Y,\rho)$ be two $(G,\mu)$ spaces. A measurable map
$\pi :  (X, \nu) \to (Y,\rho)$ is called {\em a factor map}, or {\em an extension}, depending on the view point, if it intertwines the group actions: for every
$g \in G$, $g\pi(x) = \pi(gx)$ for $\nu$ almost every $x \in X$.
\end{definition}

\begin{definition}
If $(Y,\rho)$ is a factor of $(X, \nu)$ we can decompose the measure $\nu$ as
$\nu = \int\limits_Y \nu_y d\rho (y)$, where the $\nu_y$ are
probability measures on $X$ with $\nu_y (\pi^{-1}(y)) = 1$ and the map
$y \mapsto \nu_y$ is measurable from $Y$ into the space of
probability measures on $X$, equipped with its natural Borel structure.
We say $(X,\nu)$ is a {\em measure preserving extension} of $(Y,\rho)$ if for each $g \in G,\; g\nu_y = \nu_{gy}$ for almost every $y\in Y$.
Note that a stationary system $(X,\nu)$ is measure preserving (i.e.
$g\nu =\nu$ for every $g\in G$) if and only if the extension $\pi : X \to Y$,
where the factor $(Y,\rho)$ is the trivial one point system,
is a measure preserving extension.
\end{definition}

\subsection{Topological models}

We begin this subsection with some remarks regarding stationary actions of $(G,\mu)$ on $(X, \nu)$ in the case  that $X$ is a compact metric space.  We then speak of a topological stationary system. In this case we can form the measure-valued martingale
\begin{equation*} 
\theta_n (\omega) = \xi_1\xi_2\cdots \xi_n \nu.
\end{equation*}
The martingale convergence theorem is valid also in this context by the separability of $\mathcal{C} (X)$, and so we obtain a measure-valued random variable $\theta (\omega) = \mathop{\lim}\limits_{n\to \infty} \theta_n (\omega)$.

\begin{definition}
A topological stationary system $(X,\nu)$ is {\em proximal} if with probability 1, the measure $\theta (\omega)$ is a Dirac measure: $\theta (\omega) = \delta_{z(\omega)}$.
\end{definition}

\begin{definition}
A stationary system $(X,\nu)$ is {\em proximal} if every compact metric factor
$(X', \nu')$ is proximal.
\end{definition}

\begin{definition}
Let $(X,\nu)$ and $(X', \nu')$ be two $(G,\mu)$ stationary systems,
and suppose that $X'$ is a compact metric space. We say that the
stationary system $(X', \nu')$ is {\em a topological model} for
$(X,\nu)$ if there is an isomorphism of the measure spaces $\phi :
(X,\nu) \to (X', \nu')$ which intertwines the $G$ actions.
\end{definition}

The following proposition is well known and has several proofs. We will be
content here with just a sketch of an abstract construction.

\begin{prop}
Every $(G,\mu)$ system $(X,\nu)$ admits a topological model. Moreover,
if $A \subset X$ is measurable we can find a topological model
$\phi : (X,\nu) \to (X', \nu')$ such that the set $A' =\phi(A)$ is a clopen
subset of the compact space $X'$.
\end{prop}

\begin{proof}
Choose a sequence of functions $\{f_n\}\subset L_\infty(X,\nu)$
which spans $L_2(X,\nu)$, with $f_1 = {\bf 1}_A$. Let $G_0 \subset G$ be a countable dense subgroup and let $\Acal$ be the 
% bbb
$G_0$-invariant closed unital $C^*$-subalgebra of $L_\infty(X,\nu)$ which is generated by $\{f_n\}$. We let $X'$ be the, compact metric, Gelfand space which corresponds to the $G$-invariant, separable, $C^*$-algebra $\Acal$.
Since $f_1^2 = f_1$ we also have
${\tilde{f}_1}^2 = \tilde{f}_1$, where the latter is the element of $C(X')$
which corresponds to $f_1$. Since $\tilde{f}_1$ is continuous it follows
that $A' : = \{x' :  \tilde{f}_1(x') = 1\}$ is indeed a clopen subset of $X'$
with $\tilde{f}_1= {\bf 1}_{A'}$. The probability measure $\nu'$ is the
measure which corresponds, via Riesz' theorem, to the linear
functional $\tilde{f} \mapsto \int f \, d\nu$.
\end{proof}

\begin{prop} \label{prop:pox-sat}
If $(X, \nu)$ is a proximal stationary system for $(G,\mu)$ then the action of $G$ on $(X,\nu)$ is SAT.
\end{prop}

\begin{proof}
Let $A$ be a measurable subset of $X$ with  $\nu (A) > 0$.  There is a topological model $(X', \nu')$ of $(X, \nu)$ such that $A$ is the pullback of a closed-open set $A'$ with $\nu' (A') = \nu (A)$.  As in
section \ref{sec:po} we form the martingale 
$\nu' (\xi^{-1}_n \xi^{-1}_{n-1} 
% ggg
\cdots \xi^{-1}_1 A')$ which converges to $\theta (\omega) (A) = \delta_{z'(\omega)} (A')$, since by the proximality of $(X, \nu)$, the topological factor $(X', \nu')$ is proximal. Now the latter limit is 0 or 1 and since the expectation of $\nu'(\xi^{-1}_n\xi^{-1}_{n-1}\cdots \xi_1^{-1} A)$ is $\nu(A) > 0$, there is positive probability that $z'(\omega) \in A'$.
When this happens
$$
\nu(\xi^{-1}_n \xi^{-1}_{n-1} \cdots\xi_1^{-1} A) = \nu' (\xi^{-1}_n \xi^{-1}_{n-1} 
\cdots\xi_1^{-1} A')\to 1.
$$
This proves that the action is SAT.
\end{proof}

\subsection{The structure theorem}
We now reformulate the structure theorem (theorem 4.3) of \cite{FG} to suit our needs.  (The theorem in \cite{FG} gives  more precise information.)

\begin{theorem} \label{thm:structure}
Every stationary system is a factor of a stationary system which is a measure preserving extension of a proximal system.
\end{theorem}
Alternatively, in view of proposition \ref{prop:pox-sat}:
\begin{theorem}
If $(X,\nu)$ is a stationary action of $(G,\mu)$, there is an extension 
$(X^*, \nu^*)$ of $(X,\nu)$ which is a measure preserving extension of an SAT action of $G$ on a stationary space $(Y,\rho)$.
\end{theorem}
This is the basic structure theorem which we will use to deduce a general multiple recurrence result for stationary actions.

\section{Multiple Recurrence for Stationary Actions}

We recall the terminology of \cite{FG}:
\begin{definition}
A $(G,\mu)$ stationary action of on $(X,\nu)$ is
{\em standard} if $(X,\nu)$ is a measure preserving extension of a
proximal action.
\end{definition}

Since proximality implies SAT we can extend this notion and
replace ``proximal" by SAT.
Theorem \ref{thm:structure} asserts that every stationary action has a standard extension.  The nature of recurrence phenomena is such that if such a phenomenon is valid for an extension of a system, it is valid for the system.  Precisely, if $\pi:X\to X'$ and $A'\subset X'$ and for the pullback $A=\pi^{-1}
(A')$ and a set $g_1, g_2, \dots, g_k$, we have
$\nu (g_1^{-1} A\cap g_2^{-1} A\cap \dots \cap g_k^{-1} A) > 0$, then $\nu'(g_1^{-1} A' \cap g_2^{-1} A'\cap \dots \cap g_2^{-1} A')> 0$.
It follows now from theorem \ref{thm:structure} that for a general multiple recurrence theorem for stationary actions, it will suffice to treat standard actions. Using the definition of a standard action we will take advantage of the multiple recurrence theorem proved in Section \ref{sec:SAT} for SAT actions and show that this now extends to any standard action.  For this we use a lemma which is based on Szemer\'edi's theorem.  By the latter there is a function $N(\delta,\ell)$, for $\delta > 0$ and $\ell$ a natural number, so that for $n \ge N(\delta, \ell)$, if $E \subset  \{ 1, 2, 3, \dots , n\}$ with $|E|\ge \delta n$ then $E$ contains an $\ell$-term arithmetic progression. We now have

\begin{lemma} \label{lem:sam}
In any probability space $(\Omega, P)$, for $n \ge N(\delta, \ell)$,
if $A_1, A_2, \allowbreak \dots, A_n$ are $n$ subsets of $\Omega$ with $P(A_i)>\delta$ for
$i = 1, 2, \dots, n$, then there exist $a$ and $d$ so that
\begin{equation*} P(A_{a} \cap A_{a+d} \cap A_{a+d}\cap \dots \cap A_{a+(\ell -1)d})> 0.
\end{equation*}
\end{lemma}

\begin{proof}
Set $f_i(x) = 1_{A_i} (x), \; i = 1, 2, \dots, n$, and let $E(x) =
\{ i: f_i (x) = 1\}$.  $|E(x)| = \sum\limits^n_{i=1} f_i (x)$ and
the condition $|E(x)| > \delta n$ is implied by $F(x) = \Sigma f_i
(x) > \delta n$. But $\int F(x)\, dP (x) = \Sigma P(A_i) > \delta
n$ and so for some set $B \subset \Omega$ with $P(B) > 0,\; F(x) >
\delta n$.
Thus for each $x \in B$ we have $|E(x)| > n\delta$ and there is an
$\ell$-term arithmetic progression $R_{a, d}(x)\subset E(x)$, so
that $x$ lies in the intersection of the $A_r$, as $r$ ranges over
the arithmetic progression $R_{a,d}(x)$.
There being only finitely many progressions we obtain for one of
these $P\Big( \bigcap\limits_{r \in R_{a,d}} A_r\Big) > 0$.
\end{proof}

We will need an additional hypothesis to obtain a general multiple recurrence theorem.
\begin{definition} A group $G$  is OU (order unbounded) if for any integer $n$ we have for some $g \in G, \; g^n\neq id$.
\end{definition}

For an OU group we can find, for any given $k$, elements $\sig \in G$
so that none of the powers $\sig, \sig^2,\dots \sig^k$  give the identity.
Note that in our proof of multiple recurrence for SAT actions,
theorem \ref{thm:SAT}, we obtain, for any subset
$A \subset X$ of positive measure, an element $\id \neq \ga \in G$ with:
\begin{equation*} \nu(A\cap \gamma^{-1} A\cap \gamma^{-2} A \cap \dots \cap \gamma^{-k} A)>0,
\end{equation*}
where for an OU group we can demand that each
$\gamma^j \neq id,\; j = 1, 2, \dots, k$. In fact, in that proof we show that
the element $\ga$ can be found within the conjugacy class of any
non-identity element $\sig$ of $G$.

We can now prove

\begin{theorem} Let $(X, \nu)$ represent a stationary action of $(G, \mu)$ with the elements of $G$ acting on $(X,\nu)$ by non-singular transformations and where $G$ is an OU group.  Let $A\subset X$ be a measurable set with $\nu(A) > 0$ and let $k \ge 1 $ be any integer then 
% bbb 
there exists an element $\gamma$ in $G$ with
$\gamma^j \neq id,\ j=1,2,\dots,k$
and with 
$$
\nu(A\cap \gamma^{-1} A\cap \gamma^{-2} A\cap \dots \cap \gamma^{-k} A) > 0.
$$
\end{theorem}

\begin{proof} 
We can assume $(X,\nu)$ is a measure preserving extension of
$(Y, \rho)$ where the action of $G$ on $(Y, \rho)$ is SAT.  Let $\pi: X\to Y$ and decompose $\nu = \int \nu_y d\rho(y)$.  Let $A\subset X$ be given and let $\delta > 0$ be such that $B = \{ y:\nu_y(A) > \delta\}$ has positive measure.
Set $N=N (\delta, k)$ in theorem \ref{thm:SAT} and find $\gamma $ with $\gamma^j \neq id$\, for $j = 1, 2, \dots, N$, and with
\begin{equation*} \rho(B\cap \gamma^{-1} B\cap \gamma^{-2} B \cap \dots \cap \gamma^{-N}B) > 0.\end{equation*}
For $y\in B \cap \gamma^{-1} B\cap \dots 
% ggg
\cap \gamma^{-N} B$ and
$j = 1, 2, \dots, N$ we will have
$$
\nu_y (\gamma^{-j} A) = \gamma^j\nu_y (A) = \nu_{\gamma^jy} (A) > \delta.
$$
We now use lemma \ref{lem:sam} to obtain for each $y\in B\cap
\gamma^{-1} B\cap \dots \cap \gamma^{N}B$ a $k$-term arithmetic
progression $R$ of powers of $\gamma$ with $\nu_y
\big(\bigcap\limits_{j\in R} \gamma^{-j}A\big) > 0$. In particular
$\bigcap\limits_{j \in R} \gamma^{-j} A\neq \emptyset$ for some
arithmetic progression $R = \{ a, a+d, a+2d,\dots, a+(k-1)d \}$ so
that with $\gamma' = \gamma^{d}$,
$$
A \cap \gamma'^{-1}  A \cap \gamma'^{-2} A\cap \dots
\cap \gamma'^{-k} \neq \emptyset.
$$
Obtaining a non-empty intersection suffices to obtain an intersection of positive measure, and so our theorem is proved.
\end{proof}

%%%%%%%%%%%%%%%%%%%%%%%%%%%%%%%%%%%%%%%%%%%%%%%%%%%%%%%%%%%%%%%%%%%%%%%%%%%%%%%%
%%%%%% %%%%%            Bibliography                            %%%%%%%%%%
%%%%%%%%%%%%%%%%%%%%%%%%%%%%%%%%%%%%%%%%%%%%%%%%%%%%%%%%%%%%%%%%%%%%%%%%%%%%%%%%

\bibliographystyle{amsplain}

\end{document}